\documentclass[11pt]{amsart}
\usepackage{color, amsmath,amssymb, amsfonts, amstext,amsthm, latexsym,mathtools}

\usepackage[active]{srcltx}
\allowdisplaybreaks

\usepackage{upgreek}
\usepackage{mathrsfs}

\usepackage{ bbold }
\usepackage{graphicx}
\usepackage[show]{ed} 
\usepackage[T1]{fontenc}
\usepackage[english]{babel}
\usepackage{enumerate}
\usepackage{babel}
\usepackage[pdftex,breaklinks,colorlinks, citecolor=blue,urlcolor=blue]{hyperref}
\usepackage{verbatim}
\numberwithin{equation}{section}
\usepackage{enumitem}

\usepackage[a4paper,top=30mm,bottom=30mm,inner=25mm,outer=25mm]{geometry}

\usepackage{enumerate}
\usepackage{babel}

\usepackage[pdftex,breaklinks,colorlinks, citecolor=blue,urlcolor=blue]{hyperref}
\usepackage{verbatim}
\numberwithin{equation}{section}

\newtheorem{theo}{Theorem}[section] 
\newtheorem{lem}[theo]{Lemma}
\newtheorem{corll}[theo]{Corollary}
\newtheorem{rem}[theo]{Remark}
\newtheorem{clm}[theo]{Claim}
\newtheorem{prop}[theo]{Proposition}

\usepackage{relsize}

\newcommand{\R}{\mathbb{R}}
\newcommand{\C}{\mathbb{C}}
\newcommand{\N}{\mathbb{N}}

\DeclareMathOperator{\re}{Re}
\DeclareMathOperator{\im}{Im}

\def \tomega  {\tilde{\omega}}
\def \tmu  {\tilde{\mu}}
\def \TR   {\tilde{R}}
\def \tvarphi {\tilde{\varphi}}

\def \tQ  {\tilde{Q}}

\def \Modomega { \tilde{\omega}}
\def  \Modmu       {\tilde{\mu}}
 
\def \W { \widetilde{W}}

\def \NNls {{\rm{NLS$_\Omega$ }}}

\def \nnls {{\rm{NLS$_{\R^d}$ }}}

\title[Multi-solitary wave solution to  NLS outside a strictly convex obstacle]{
Construction of multi-solitary waves solution to the focusing nonlinear Schr\"odinger equation outside an obstacle in the $L^2$-subcritical case}

 \author[O. Landoulsi]{Oussama Landoulsi }
\address{Department of Mathematics \\ University of Massachusetts \\ Amherst, MA 01003, USA}
\email{olandoulsi@umass.edu}

\date{\today}
\makeatletter
\@namedef{subjclassname@2010}{Mathematics Subject Classification}
\makeatother


\begin{document}

\begin{abstract}
We consider the focusing $L^2$-subcritical Schr\"odinger equation in the exterior of a smooth, compact, strictly convex obstacle $\Theta \subset \mathbb{R}^d$. We construct a solution that, for large times, behaves asymptotically as a finite sum of solitary waves on $\mathbb{R}^d$, each traveling with sufficiently large and distinct velocities, and satisfying Dirichlet boundary conditions. The construction is achieved via a compactness argument similar to that introduced by F.Merle in 1990 for constructing solutions of the NLS equation that blow up at several points, combined with modulation theory, the coercivity property of the linearized operator, and localized energy estimates. 

\end{abstract}
\maketitle
\hspace{2cm}

\tableofcontents

\section{Introduction}

We consider the focusing nonlinear Schr\"odinger equation in the exterior of a smooth compact strictly convex obstacle $\Theta \subset \R^d,$ for any $d \geq 1,$ with Dirichlet boundary conditions: 
\begin{equation}
\tag{NLS$_{\Omega}$}
\begin{cases}
\label{NLS} 		
i\partial_tu+\Delta_{\Omega} u = -|u|^{p-1}u  \qquad  &  (t,x)\in \R \times\Omega ,\\ 
u(t_0,x) =u_0(x)  &  x \in \Omega  , \\
u(t,x)=0  &   (t,x)\in \R \times \partial\Omega ,
\end{cases}		
\end{equation}		
     where $\Omega=\R^d \setminus \Theta$, $\Delta_{\Omega}$ is the Dirichlet Laplace operator on $\Omega$ and $t_0 \in \R$ is the initial time. Here, $u$ is a complex-valued function,
\begin{align*}
 u:  \R &\times  \Omega \longrightarrow \C \\
 	(t &, x)  \longmapsto  u(t,x) .
 \end{align*}
   We take the initial data $u_0 \in H^1_0(\Omega),$ where $H^1_0(\Omega)$ is the Sobolev space 
   $$ \{ u \in L^2(\Omega) \; \text{such that} \; \left| \nabla u \right| \in L^2(\Omega) \; \text{and} \; u_{|\partial \Omega}=0\}. $$
    
    The \NNls equation is locally well-posed in $H^1_0(\Omega),$ for $1<p<\frac{d+2}{d-2}$ see \cite{AnRa08,PlVe09,Ivanovici10,MR2683754,Ou19,BlairSmithSogge2012}. The solutions of the \NNls equation satisfy the mass and energy conservation laws: 
 \begin{align*}
 M_{_\Omega}[u(t)] &:= \int_{\Omega} |u(t,x)|^2 dx = M[u_0] ,\\
E_{\Omega}[u(t)] &:= \frac{1}{2} \int_{\Omega} |\nabla u(t,x)|^2 dx - \frac{1}{4} \int_{\Omega}\left|u(t,x)\right|^{4}dx=E[u_0].  
\end{align*}

The \nnls equation is invariant by the scaling transformation, that is, 
  \begin{equation*}
 u(t,x) \longmapsto \lambda^{\frac{2}{p-1}}u(\lambda x,\lambda^2 t) \;, \; \text{ for }  \lambda >0.
 \end{equation*}

 This scaling identifies the critical Sobolev space $\dot{H}^{s_c}_{x}$, where the critical regularity $s_c$ is given by $s_c:=\frac{d}{2} - \frac{2}{p-1}$. The case when $ s_c=0$ is referred to as mass-critical or $L^2$-critical and the case when $s_c=1$ is called energy-critical or $H^1$-critical.  \\

Throughout this paper, we will take $ 1<p<1+\frac{4}{d}$. Since the presence of an obstacle does not change the intrinsic dimensionality of the problem, we regard the nonlinear Schr\"odinger equation outside an obstacle as having the same criticality, and thus as being mass-subcritical. \\

Consider solitary waves solution of  (NLS$_{\Omega}$), with $\Omega=\R^d$, that is $u(t,x)=e^{i t \omega }Q_{\omega}(x)$ \\ 
where $ Q_\omega $ is a solution of the nonlinear elliptic equation: 
 \begin{equation}
 \label{eq_Q}
 \begin{cases}
 -\Delta Q_\omega + \omega \, Q_\omega= \left| Q_\omega\right|^{p-1} Q_\omega ,   \\
 \;  Q_\omega \in H^1(\R^d).
 \end{cases}
 \end{equation}

This elliptic equation admits solutions if and only if $\omega > 0$. In this paper, we will denote by $ Q_\omega $ the ground state which is the unique radial positive solution of \eqref{eq_Q}.  \\

We recall that $Q_{\omega}$ is smooth and exponentially decaying at infinity, i.e., $|\nabla Q_{\omega}|+|Q_{\omega}|\leq C e^{-\frac{\sqrt{\omega}}{2}|x|},$ and characterized as the unique minimizer for the Gagliardo-Nirenberg inequality up to scaling, space translation and phase shift, see \cite{Kwong89,BeLi83a,GidasNirenberg79}. \\

The \nnls equation posed on the whole Euclidean space $\R^d$ admits the following symmetries: If $u(t, x)$ is a solution, then
\begin{itemize}
    \item Space-translation invariance:  for any $x_0 \in \R^d,$ $u(t,x-x_0)$ is also a solution to the \nnls equation.
     \item Phase invariance:  for any $\mu \in \R,$ $u(t,x)e^{i \mu}$ is also a solution to the \nnls equation.
\end{itemize}
Moreover, the \nnls equation also enjoys Galilean invariance. If $u(t,x)$ is solution, then $u(t,x-vt) \, e^{i(\frac{x \cdot v}{2} -   \frac{ | v |^2   }{4} t ) } $ is also a solution, for $v \in \R^d$.\\

Applying a Galilean transform and the symmetries to the soliton solution 
$e^{i  t \omega } Q_{\omega}(x)$ of the \nnls equation, we obtain a solitary wave solution, moving on the line $x=x_0+t v$ with velocity $v \in \R^d:$ 
\begin{equation}
\label{soliton}
u(t,x) = e^{i(\frac{1}{2}(x \cdot v)-\frac{1}{4} \left| v\right|^2 t + t \, \omega +\mu)}Q_{\omega}(x-x_0-t \,v).
\end{equation} 

The soliton \eqref{soliton} is a global solution of the focusing nonlinear Schr\"odinger equation posed on the whole space, but it is not a solution of the \NNls equation. Our goal is to construct a solution to the \NNls equation satisfying Dirichlet boundary conditions and behaving asymptotically, as $t \to +\infty$, like a finite sum of solitary waves of the form \eqref{soliton}, each with sufficiently large and distinct velocities. \\

We now present the main results of this paper. \\

Assume $1<p<1+\frac 4d.$ let $K \in \N \backslash \{0 \}$ and for any $k\in \{ 1, \cdots,  K \}$, we define $\Psi_k$  to be a $C^{\infty}$ function such that:
 \begin{equation}
 \label{def-Psi}
 \begin{cases}
\Psi_k=0 & \; \text{near} \;\, \Theta , \\
\Psi_k=1 & \; if \;  \left|x\right| \gg 1 .
 \end{cases}
 \end{equation}

\begin{theo}
\label{theorem-pranp}
Let $K \in \N \backslash \{0 \}.$ For any $k\in \{ 1, \cdots,  K \},$ let $\Psi_k$ defined as \eqref{def-Psi}, and let $\omega_k >0, \; v_k \in \R^d, \; x_k^0 \in \R^d $ and $\mu_k \in \R. $ Assume that $\forall k \neq k^{'},  \;   v_{k} \neq  v_{k^{'}}. $ \\ 
Let $$R_k(t,x)= Q_{\omega_k}(x-x_k^0-t v_k) \Psi_k(x) e^{i ( \frac 12 (x \cdot v_k)- \frac{1}{4} t \left| v_{k} \right|^2 +  t \omega_k + \mu_{k} )}. $$

 Then there exists  $T_0>0$ and a solution $u(t)$ to the \NNls equation satisfying 
 
 $$\left\|u(t) -\sum_{k=1}^{K} R_k(t)  \right\|_{H^1 _0(\Omega)} \leq C  e^{- \sigma_0  t }   \qquad \forall t \in [T_0,+\infty), $$ 
for some constant $\sigma_0 >0 $ and $C >0.$
\end{theo}

The structure of the proof of Theorem \ref{theorem-pranp} follows the approach used in the construction of multi-solitary wave solutions in the Euclidean setting, namely \cite{MartelMerle06} for the $L^2$-subcritical case and \cite{MR2815738} for the $L^2$-supercritical case. In Section~\ref{sec:construct-solution}, we assume the existence of a solution $u_n$ of the \NNls equation for $t \in [T_0, T_n]$ that satisfies a uniform estimate with initial data $u_n(T_n)=\sum_{k=1}^{K} R_k(T_n)$. We then apply a compactness argument to construct a solution $u$ to the \NNls equation on $[T_0,+\infty)$, which concludes the proof of Theorem \ref{theorem-pranp}. The compactness argument employed here is similar to the main idea introduced by F.~Merle in 1990 for the construction of solutions blowing up at several points in the critical case $p = 1 + \tfrac{4}{d}$, see also \cite{MR3124722,MR2228459,MartelMerleTai02}. In the subcritical case, it is well known that solitary waves are stable; see \cite{CaLi82} for the concentration-compactness approach and \cite{MR820338} for a different method based on the expansion of conservation laws. From \cite{MR820338}, there exists $\lambda > 0$ such that for any real-valued function $h \in H^1$, 
\begin{equation}
\label{coerv-subcrit}
\displaystyle (h,Q_\omega)=(h,\nabla Q_\omega)=0 \Longrightarrow \displaystyle\int \{\left| \nabla h\right|^2 + \omega \left| h\right|^2 - p Q_\omega^{p-1} \left|h \right| ^2  \} \geq \lambda \left\| h \right\|_{H^1}^2 .
\end{equation}
In Section~\ref{sec:uniformestimate}, we prove the existence of the solution $u_n$ together with the uniform estimate assumed in the previous section. To this end, we employ a modulation in the phase, translation, and scaling parameters in the decomposition of the solution for large times in order to obtain orthogonality conditions. We then apply a bootstrap argument to control the modulation parameters, combined with energy estimates and localization techniques for the coercivity estimate of the linearized operator.
Note that, even though we use some similar arguments, a large part of the proof of Theorem~\ref{theorem-pranp} is different. This is due to of the presence of the obstacle $\Theta,$ which makes the calculations more complicated.  \\ 

\begin{rem}
Theorem \ref{theorem-pranp} can be extended to general nonlinearity $f(|u|^2)u,$ with $f$ of Class $\mathcal{C}^1$ satisfying $f(0)=0$ and 
\begin{align*}
  \forall s \geq 1, \quad |f^{\prime}(s^2)| < C s^{p-2}, \quad \text{ for some } p < \frac{d+2}{d-2}. 
\end{align*}
Moreover, the soliton \eqref{soliton} is a solution where $Q_{\omega}$ is a solution of the following elliptic equation 
\begin{align*}
 - \Delta Q_{\omega} + \omega Q_{\omega} = f(Q_{\omega}^2) Q_{\omega}, \quad \text{ for } \omega>0.    
\end{align*}
A similar stability results holds for the general nonlinearity $f(|u|^2)u,$ see \cite{MR820338, CaLi82}. There exists $\Lambda>0$ such that for any real-valued function $h \in H^1$, 
\begin{equation}
\label{eq:GF-coer}
\displaystyle (h,Q_\omega)=(h,\nabla Q_\omega)=0 \Longrightarrow \displaystyle\int \{\left| \nabla h\right|^2 + \omega \left| h\right|^2 -  (f(Q_\omega^2) + 2 Q_\omega^2 f^{\prime}(Q_\omega^2) )|h|^2 \}  \geq \lambda \left\| h \right\|_{H^1}^2 .
\end{equation}
For the pure power nonlinearity, the above coercivity property is satisfied when $f(s^2) = s^{p-1}$ for $1 < p < 1 + \tfrac{4}{d}$; see \cite{Maris02,McLeod93}. We provide the proof in the pure power case, which can be adapted with slight modifications to the general nonlinearity. In particular, we highlight the differences or invoke the general nonlinearity when necessary throughout the paper.
\end{rem}

\begin{rem}
Observe that the solution constructed in Theorem \ref{theorem-pranp} is non–dispersive (it exists for all large times and does not scatter) in the sense that, by strong $H^1$ convergence, we have  
\[
\int_{\Omega} |u(t,x)|^2 \, dx = \sum_{k=1}^{K} \int_{\Omega} |R_k(t,x)|^2 \, dx, 
\qquad 
E(u(t)) = \sum_{k=1}^{K} E(R_k(t)).
\]
This means that the entire mass and energy of the solution are carried by the solitary wave components. In contrast, for a general solution, part of the $L^2$ norm is dispersed over time due to the equation’s dispersive nature. 
 
Note that the nonlinear Schr\"odinger equation is time reversible, i.e., if $u(t,x)$ is a solution then $\bar{u}(-t,x)$ is also a solution, therefore the result of Theorem \ref{theorem-pranp} can be stated in similar way for $t \to -\infty.$ 
\end{rem}
The analysis of dispersive equations in exterior domains, motivated by understanding the influence of the geometry on the dynamics of solutions, since the late 1950s. Early contributions concern wave-type models with Dirichlet or Neumann boundary conditions in the exterior of an obstacle. In 1959, Wilcox \cite{Wilcox59} investigated the decay of linear waves outside a sphere. Morawetz later extended this to star-shaped obstacles \cite{Morawetz61}, and in joint work with Ralston and Strauss, to the general class of non-trapping geometries \cite{MoraRalstonStraussCorec78}. \\

In the case of the nonlinear Schr\"odinger equation, the Cauchy problem in $H^1_0(\Omega)$ was first studied by Burq, Gérard, and Tzvetkov \cite{BuGeTz04a},  for sub-cubic nonlinearities ($p<3$) under non-trapping assumptions. Subsequent developments include Anton’s result for the cubic case \cite{AnRa08}, the work of Planchon and Vega \cite{PlVe09} in the energy-subcritical regime ($1<p<5$) in $d=3$, and the energy-critical theory ($p=5$) of Planchon and Ivanovici \cite{MR2683754}. Related results for convex obstacles appear in \cite{BlairSmithSogge2012, Ivanovici07, Ivanovici10, DongSmithZhang2012}. Local well-posedness in the critical Sobolev space was first obtained in \cite{MR2683754} for $3+\tfrac{2}{5}<p<5$, and later extended to $\tfrac{7}{3}<p<5$ in \cite{Ou19}  using the fractional chain rule developed in \cite{killip2015riesz}.  Let us also mention the recent work on dispersive estimates outside non-trapping obstacles \cite{ThomasYang24}. \\

The nonlinear dynamics of the focusing {\rm NLS}$_\Omega$ equation outside an obstacle was studied in \cite{killip2015riesz}, where the authors proved that solutions below a natural mass–energy threshold scatter to a linear solution. In \cite{Ou19}, we constructed a solution $u(t)$ of the focusing {\rm NLS}$_\Omega$ equation outside an obstacle in the $L^2$-supercritical case. This solution demonstrates the optimality of the threshold for global existence and scattering obtained in \cite{killip2015riesz}.  The existence of blow-up solutions with negative energy was initiated in \cite{OL22-blow-up}, where a new modified variance quantity was introduced to establish the existence of blow-up solutions and a blow-up criterion under the mass–energy threshold. The dynamics at the mass–energy threshold were further studied in \cite{ThomasOSvetlana22}, where it was shown that all solutions at the threshold are globally defined and scatter in $H^1_0(\Omega)$ in both time directions.  The behavior of solitary waves after interaction or collision with an obstacle was numerically investigated in \cite{OsvetlanaKai23}. The authors distinguished two types of interaction and showed how the presence of the obstacle affects the overall behavior of the solution after the interaction.

\section{Construction of the solution}
\label{sec:construct-solution}
This section is devoted to the proof of Theorem \ref{theorem-pranp}, relying on a uniform estimate that will be derived in the next section. 

Let $$R_k(t,x)= Q_{\omega_k}(x-x^{0}_{k}-t v_k) \Psi_k(x) e^{i ( \frac 12 (x \cdot v_k)- \frac{1}{4} t \left| v^{k} \right|^2 + \omega_k t + \mu_{k} )}, $$ be as defined in Theorem \ref{theorem-pranp}, which behave asymptotically  as a solitary wave solution to the \nnls equation in the whole Euclidean space, see \cite{Ou19}.

The solution $u(t)$ described in Theorem \ref{theorem-pranp} is obtained through an asymptotic argument. Consider an increasing sequence of times $(T_n)_{n \geq 1}$ such that $T_n \to \infty$ as $n \to \infty$. For $n \geq 1,$ we denote by $u_n$ the unique global $H^1$ solution of
\begin{equation}
\label{eq:defun}
\begin{cases}
i \partial_t u_n + \Delta u_n =-|u_n |^{p-1} u_n,  \\
u_n(T_n)=R(T_n), 
\end{cases}
\end{equation}
where $ R(t):= \displaystyle \sum_{k=1}^K R_k(t).$ Next, to establish Theorem \ref{theorem-pranp}, we claim the following uniform estimate.

\begin{prop}[Uniform estimate]
\label{prop:uniform-estimate}
There exist $n_0 \geq 0, \; T_0>0, \; C_0 >0$ and $\sigma_0>0$ such that for all $n \geq n_0,$ the solution $u_n$ of \eqref{eq:defun} is well-defined on the time interval $[T_0, T_n]$ and satisfies, 
\begin{equation}
\label{eq:uniform-estimate}
\left\|  u_n(t)-R(t) \right\|_{H^1_0(\Omega)} \leq C_0 e^{-\sigma_0 t }.   
\end{equation}    
\end{prop}

\begin{proof}
    We will assume this proposition to prove Theorem \ref{theorem-pranp} postponing the proof of it to Section~\ref{sec:uniformestimate}.
\end{proof}

We now begin the proof of Theorem \ref{theorem-pranp}, assuming the main uniform estimate provided in Proposition \ref{prop:uniform-estimate}. The argument is based on a compactness method combined with the uniform estimate \eqref{eq:uniform-estimate}. By suitably relabeling the indices, we may set $n_0 = 0$ in Proposition \ref{prop:uniform-estimate}.\\

\begin{proof}[Proof of Theorem $\ref{theorem-pranp}$ assuming Proposition $\ref{prop:uniform-estimate}$] 
The proof proceeds in several steps.
\begin{itemize}
\item Step 1: Compactness argument.
The Proposition $\ref{prop:uniform-estimate}$ implies that there exists a sequence $u_n(t)$ of solution defined on $[T_0,T_n]$ such that 

 $$ \forall n \in \N, \; \forall t \in [T_0,T_n], \quad   \|u_n(t) - R(t) \|_{H^1_0(\Omega)} \leq C_0 e^{-\sigma_0 t } . $$

\begin{lem}
\label{limsup}
\begin{equation*}
\displaystyle \lim_{M\rightarrow +\infty} \sup_{\,n\in \mathbb{N}}  
\int_{|x| \geq M}  u_n^2(T_0,x)  \, dx =0 .
\end{equation*}
\end{lem}
\begin{proof} 

The proof of this lemma proceeds along the same lines as in \cite{MR2271697}, for the multi-soliton solutions of \nnls equation in the subcritical case, and as in \cite{Ou19}, which treats the construction of solitary wave like-solutions for the \NNls equation. We provide the details here for the sake of completeness of this paper. 

Let $\varepsilon > 0 $ and $T_{\varepsilon} \geq T_0 $ such that: $C_0^2 e^{ -2\sigma_0 T_{\varepsilon} }< \varepsilon$, where  $C_0$ and $ \sigma_0 $  are the same constants as in the Proposition \ref{prop:uniform-estimate}. For $n$ sufficiently large such that $T_n \geq T_{\varepsilon}$, and in view of \eqref{eq:uniform-estimate}, we have

\begin{equation*}
    \int_{\Omega} \left|u_n(T_{\varepsilon}) - R(T_{\varepsilon})  \right|^{2} dx \leq C_0^2 e^{ -2\sigma_0 t }< \varepsilon.
\end{equation*}

Let $M(\varepsilon) > 0$ such that 

$$ \int_{|x| \geq M(\varepsilon)} \left| R(T_{\varepsilon})  \right|^{2} \, dx < \varepsilon ,$$ 
by direct computation,

\begin{equation*}
\int_{\left|x\right| \geq M(\varepsilon)} \left| u_n(T_{\varepsilon}) \right|^2  dx \leq 4 \varepsilon .
\end{equation*}

Now consider a $C^1$ cut-off function $f:\R \longrightarrow [0,1]$ such that \\
\begin{equation*}
f \equiv 0 \; \; \text{ on}  \;  ]-\infty,1]; \quad \qquad  0<f'<2 \; \;  \text{on} \; (1,2); \qquad \qquad f \equiv 1\; \; \text{ on }\; (2,+\infty).
\end{equation*}

For $K_{\varepsilon}>0$ to be specified later, we can check that
\begin{equation}
\label{eq_u_n-f}
\frac{d}{dt} \int_{\Omega} \left|u_n(t)  \right|^2 f\left(\frac{\left|x \right|- M(\varepsilon)}{K_{\varepsilon} } \right)dx = \frac{-2}{K_{\varepsilon}} \im \int_{\Omega} u_n(t) \left(\nabla \overline{u_n}. \frac{x}{\left|x\right|}\right)  f'\left(\frac{\left|x\right|-M(\varepsilon)}{K_{\varepsilon}} \right)dx.
\end{equation}

From Proposition \ref{prop:uniform-estimate}, there exists $  \alpha > 0,$  $\forall n$ and $\forall t \geq T_0, \;  \left\|u_n(t) \right\|_{H^1_0}^2  \leq \alpha.$ Using $\eqref{eq_u_n-f}$ we get 
\begin{align*}
\left|\frac{d}{dt}  \int_{\Omega} \left| u_n(t) \right|^2  f\left(\frac{\left|x\right|- M(\varepsilon)}{K_{\varepsilon}}\right) \right|  \leq \frac{4}{K_{\varepsilon} }\left\| u_n(t) \right\|_{H^1_0}^2 \leq \frac{ 4 }{K_{\varepsilon}}  \, \alpha .
\end{align*}
Now, we choose $K_{\varepsilon} > 0$ independently of $n$ such that $ K_{\varepsilon}  \geq  \left(\frac{T_\varepsilon- T_0}{\varepsilon}  \right) 4  \alpha , $
which yields

$$\left|\frac{d}{dt} \int_{\Omega} \left|u_n(t) \right|^2 f\left(\frac{\left|x\right|-M(\varepsilon)}{K_\varepsilon} \right) \right| \leq \frac{\varepsilon}{T_\varepsilon-T_0} \,. $$
Integrating on the time interval $\left[T_0,T_\varepsilon\right]$, we get
\begin{align*}
& \int_{\Omega} \left|u_n(T_0) \right|^2 f\left(\frac{\left|x\right|-M(\varepsilon)}{K_\varepsilon}\right) dx  - \int_{\Omega} \left| u_n(T_\varepsilon) \right|^2 f\left(\frac{\left|x\right|-M(\varepsilon)}{K_\varepsilon}\right)  dx \\ &\leq \int_{T_0}^{T_\varepsilon} \left| \frac{d}{dt} \int \left| u_n(t) \right|^2 f\left(\frac{\left|x\right|-M(\varepsilon)}{K_\varepsilon}\right)  dx \right| dt  \\ & \leq \varepsilon .
\end{align*}
Hence, 
$$ \int_{\Omega} \left| u_n(T_0) \right|^2 f\left(\frac{\left|x\right|-M(\varepsilon)}{K_\varepsilon}\right) dx \leq \varepsilon + \int_{\Omega} \left|u_n(T_\varepsilon) \right|^2  f\left(\frac{\left|x\right|-M(\varepsilon)}{K_\varepsilon}\right) dx . $$

Due to the properties of $f$, we have

\begin{align*}
\int_{\left|x\right|>2 K_{\varepsilon}+ M(\varepsilon) } \left|u_n(T_0)\right|^2 dx &\leq \int_{\Omega} \left|u_n(T_0)\right|^2 f\left(\frac{\left|x\right|-M(\varepsilon)}{K_\varepsilon}\right) dx  \\  & \leq  \varepsilon +\int_{\Omega} \left|u_n(T_\varepsilon) \right|^2  f\left(\frac{\left|x\right|-M(\varepsilon)}{K_\varepsilon}\right) dx \\ & \leq \varepsilon + \int_{\left|x\right| \geq M(\varepsilon)} \left| u_n(T_\varepsilon) \right|^2 dx \\ & \leq \varepsilon + 4 \varepsilon= 5 \varepsilon . 
\end{align*}

This concludes the proof of the lemma.
\end{proof}

By the main Proposition  \ref{prop:uniform-estimate}, we have 
 $$ \left\|u_n(T_0)\right\|_{H^1_0(\Omega)} \leq \alpha . $$
Since $H^1_0$ is a Hilbert space, there exists a subsequence of $(u_n(t))_n$ that we still denote by $(u_n(t))_n$ to simplify notation and $\mathcal{U}_0 \in H^1_0(\Omega) $ such that 
$$ u_{n}(T_0) \rightharpoonup \mathcal{U}_0  \quad \text{in} \; \; H^1_0(\Omega),  \quad   \text{as} \; \; n \longrightarrow +\infty .$$

By the compactness of the embedding of $H^1_0(\{ \left|x \right| \leq A \})$ into $L^2(\{ \left|x \right| \leq A \})$, we have 
$$u_n(T_0) \longrightarrow \mathcal{U}_0 \quad \text{ in} \quad L^2_{loc}(\Omega) \, .$$

By Lemma $\ref{limsup}$, we get 
\begin{align}
    \label{eq:convL2}
    u_n(T_0) \longrightarrow \mathcal{U}_0 \quad \text{ in} \quad L^2(\Omega).
\end{align}
\end{itemize}
\begin{itemize}

\item Step 2: Construction of the solution. 
Recall that the \NNls equation is well-posed in $L^2(\Omega)$. Let $\widetilde{u}$ be the maximal solution of\begin{equation}
\begin{cases}
i\partial_t \widetilde{u}+\Delta \widetilde{u}= -|\widetilde{u}|^{p-1}\widetilde{u} \qquad  \forall (t,x)\in [T_0,\widetilde{T})\times\Omega, \\  \widetilde{u}(T_0,x) =\mathcal{U}_0  \qquad\qquad \qquad\forall x \in \Omega , \\
	\widetilde{u}(t,x)=0 \qquad\qquad \qquad  \quad \forall(t,x)\in[T_0,\widetilde{T})\times \partial\Omega . 
\end{cases}
\end{equation}

For $n$ large enough, $ u_n(t)$ is defined for all $ t \in [T_0,\widetilde{T}).$ 
Then by the continuity dependence of initial data upon initial data in $L^2,$ and \eqref{eq:convL2}  we have
$$u_n(t) \longrightarrow \widetilde{u}(t)  \quad \text{in} \quad L^2(\Omega). $$

From Proposition \ref{prop:uniform-estimate}, we know that for $n$ large enough $u_n(t)$ is uniformly bounded in $H^1_0(\Omega)$. Then necessarily, 

$$ \forall t \in [T_0,\widetilde{T}), \quad u_n(t) \rightharpoonup \widetilde{u}(t) \quad \text{ in } \; \,  H^1_0(\Omega).  $$

Using the property of weak convergence and by the main proposition, it follows that 

$$\forall t \in [T_0,\widetilde{T}), \quad  \left\|\widetilde{u}(t)- R(t) \right\|_{H^1_0(\Omega)} \leq \lim \inf \left\|u_n(t)-R(t)\right\|_{H^1_0(\Omega)} \leq C_0 e^{- \sigma_0 t } .$$

In particular, we deduce that, $\widetilde{u}$ is bounded in $H^1_0(\Omega)$. From the blow-up alternative, we get $\widetilde{T}=+\infty$. Finally, we have $\widetilde{u} \in C ([T_0,+\infty), H^1_0(\Omega))$  and by \eqref{eq:uniform-estimate}, we obtain

$$\forall t \in [T_0,+\infty), \quad \left\| \widetilde{u}(t)-R(t)\right\|_{H^1_0(\Omega)} \leq e^{-\delta \sqrt{\omega} |v| t },$$ 
which concludes the proof of Theorem \ref{theorem-pranp}.
\end{itemize}\end{proof}

\section{Proof of the uniform estimate}
\label{sec:uniformestimate}
In this section, we establish the uniform estimate as a consequence of the following bootstrap result.
\begin{prop}
\label{prop:bt}
    There exist $n_0 >0, \, T_0>0, \, \alpha_0 >0, \sigma_0>0$ such that for all $n \geq n_0$ and for all $\bar{t} \in [T_0,T_n],$ if 
\begin{align}
\label{eq:bt1}
\forall t \in [\bar{t}, T_n], \; \left\| u_n(t) - R(t)\right\|_{H^1_0(\Omega)} \leq \alpha_0 e^{-\sigma_0 t } ,        
\end{align}
then \begin{align}
\label{eq:bt2}
\forall t \in [\bar{t}, T_n], \; \left\| u_n(t) - R(t)\right\|_{H^1_0(\Omega)} \leq \frac{\alpha_0}{2} e^{-\sigma_0 t } .
\end{align}
\end{prop}

Next, we verify that Proposition \ref{prop:uniform-estimate} follows from a standard continuity argument in $H^1_0$ combined with Proposition \ref{prop:bt}.  
Recall that the mapping $t \mapsto u_n(t) \in H^1_0(\Omega)$ is continuous for all $n \in \mathbb{N}$.  
Let $n_0, T_0,$ and $\alpha_0$ be as given in Proposition \ref{prop:bt}.  
Let $n \geq n_0$, since $u_n(T_n) = R(T_n)$, there exists a sufficiently small $\tau_1 > 0$ such that \eqref{eq:bt1} holds on the interval $[T_n - \tau_1, \, T_n]$. \\  
Define
\begin{align*}
  \bar{t}:= \inf\{ t \in [T_0,T_n], \; \text{ such that for all } \, s \in [t,T_n], \; \eqref{eq:bt1} \, \text{holds}  \}.  
\end{align*}\
We claim that $\bar{t} = T_0$.  Indeed, assume by contradiction that $\bar{t} > T_0$.  Then, by Proposition \ref{prop:bt}, the estimate \eqref{eq:bt2} holds on the interval $[\bar{t}, T_n]$.  
By continuity, there exists a sufficiently small $\tau_2 > 0$ such that \eqref{eq:bt2} also holds on $[\bar{t} - \tau_2, T_n]$, contradicting the definition of $\bar{t}$.  
Therefore, we must have $\bar{t} = T_0$.  Consequently, for all $n \geq n_0$, the uniform estimate \eqref{eq:uniform-estimate} is valid on $[T_0, T_n]$ with $C_0 = \alpha_0$.  
This completes the proof of Proposition \ref{prop:uniform-estimate}, under the assumption that Proposition \ref{prop:bt} holds.\\

\textbf{Proof of Proposition \ref{prop:bt}.} The rest of this section is devoted to the proof of Proposition \ref{prop:bt}. Before proceeding with the proof, observe that since $T_n \to +\infty$ as $n \to \infty$ and $T_0$ is fixed, Proposition \ref{prop:bt} can be viewed as a stability-type result for large time. In contrast, establishing Proposition~\ref{prop:bt} is more straightforward than proving a full stability result for multi-soliton solutions, see \cite{MR2228459}.  \\

  First, we claim the following result for the velocity vectors. 
\begin{clm}
  Let $(v_k)$ be $K$ vectors in $\mathbb{R}^d$ such that for any $k \neq k'$, $v_k \neq v_{k'}$.  
Then, there exists an orthonormal basis $(e_1, \ldots, e_d)$ of $\mathbb{R}^d$ such that for any $k \neq k'$,  
\[
(v_k, e_1) \neq (v_{k'}, e_1).
\]
  
\end{clm}
\begin{proof}
    This is an elementary geometric property of $\mathbb{R}^d$.  
Let $k, k' \in \{1, \ldots, K\}$ with $k \neq k'$.  
The set of unit vectors $e_1 \in \mathbb{R}^d$ satisfying  
\[
\langle v_k - v_{k'}, e_1 \rangle = 0
\]  
has Lebesgue measure zero on the unit sphere.  
Therefore, we can choose a unit vector $e_1 \in \mathbb{R}^d$ such that for all $k \neq k'$,
\[
\langle v_k - v_{k'}, e_1 \rangle \neq 0.
\]
Finally, we complete $e_1$ to any orthonormal basis $(e_1, \ldots, e_K).$ 
\end{proof}
Since the \NNls  equation is invariant by rotation, we assume without any restriction, that the direction $e_1$ given in the Claim above is $x_1.$ Therefore, we may assume for any $k\neq k',$ $v_{k,1} \neq v_{k',1}.$ In fact we suppose that $$ v_{1,1}<v_{2,1}<\cdots< v_{K,1}.$$
 
Next, throughout this paper we define $\sigma_0 >0$ such that 
\begin{align}\label{eq:defSig0}
\sqrt{\sigma_0}:=\frac{1}{16} \min \left( v_{2,1}-v_{1,1}, \cdots, v_{k,1}-v_{K-1,1}, \sqrt{\omega_1},\cdots,  \sqrt{\omega_K}  \right). 
\end{align}

Let $\alpha_0>1$ and $T_0$ large enough to be determined later. Let $n_0>0$ be such that $T_{n_0}>T_0.$ In the following, we will
drop the index $n$ for most variables. Hence, we will write u for $u_n$, etc. \\

Assume that \eqref{eq:bt1} holds on $[\bar{t}, T_n]$ for some $\bar{t} \in [T_0,T_n],$ i.e., 
\begin{align*}
\forall t \in [\bar{t}, T_n], \; \left\| u_n(t) - R(t)\right\|_{H^1_0(\Omega)} \leq \alpha_0 e^{-\sigma_0 t } .        
\end{align*}

\subsection{Modulation final data} 
First, we use some modulation in the scaling, phase and translation parameters in the decomposition of the solution in a sum of solitary waves to obtain the orthogonality conditions. Next, we use a bootstrap argument
to control these parameters and some scalar product that are related to the size of the solitary waves.

\begin{lem}
\label{lem:modu-static}
There exist $C,\epsilon >0$ such that the following holds: For $k=1,\cdots,K,$ given $\alpha_k \in \R^d,$ such that  $ \min \{|\alpha_i-\alpha_j |, i \neq j \} \geq \frac{1}{\epsilon},$ $\omega_k \in \R, \theta_k \in \R $ and  $\mu_k \in \R,$  if $$ \left\| u-R \right\|_{H^1_0} \leq \epsilon, $$
then there exist $Y:=(y_k)_{1\leq k \leq K} \in (\R^d)^K,$  $\W=( \Modomega_k)_{1 \leq k \leq K} > 0,$ and 
$\Modmu:=(\Modmu_k)_{1 \leq k \leq K} \in (\R)^K,$ such that setting $$h(x)= u(x)-\tilde{R}(x),$$
the following hold,
\begin{align*}
  \left\|h \right\|_{H^1_0} + \sum_{k=1}^{K} \left| \Modmu_k - \mu_k \right|+ \sum_{k=1}^{K}  |y_k|+ \sum_{k=1}^{K}  | \Modomega_k-\omega_k|  \leq C \left\| u-R \right\|_{H^1_0},
\end{align*}
and \begin{align*}
&\re \int_{\Omega} \tilde{R}_k(x) \bar{h}(x) dx = 0, \qquad 
\im \int_{\Omega} \tilde{R}_k(x) \bar{h}(x) dx =0, \;  \\
& \re \int_{\Omega} \nabla( \tilde{Q}_{\omega_k}(x)) \Psi_k(x) e^{i \tilde{\varphi}_k(x)} \bar{h}(x) dx=0, \qquad \text{ for } k=1,\cdots, K,
\end{align*}
Where, \begin{align*}
&R(x)=\sum_{k=1}^K R_k(x), \quad R_k(x)=Q_{\omega_k}(x-\alpha_k)\Psi_k(x) e^{i \varphi_k(x)}, \text{with} \quad  \varphi_k(x):=\frac 12 (x\cdot v_k)+ \theta_k+\mu_k, \\ 
&\tilde{R}(x)=\sum_{k=1}^K \tilde{R}_k(x), \quad \tilde{R}_k(x)=\tilde{Q}_{\Modomega_k}(x) \Psi_k(x) e^{i \tilde{ \varphi_k  }(x) }, \text{with} \quad \tilde{ \varphi_k  }(x):=\frac 12 (x\cdot v_k)+ \theta_k+\Modmu_k ,\\
&\tilde{Q}_{\Modomega_k}(x):=Q_{\Modomega_k}(x-\alpha_k-y_k).
\end{align*}
\end{lem}
\begin{proof}
The proof follows from a standard application of the Implicit Function Theorem. For brevity, we only outline the main steps of the argument and omit the details, see \cite[Lemma 3.2]{Ou19} for more details. \\
Let
\begin{align*}\Phi: (0,\infty)^K \times (\R^d)^K \times \R^K \times H^1_0(\Omega)& \longrightarrow \R^K \times \R^K \times (\R^d)^K, \\ 
(\W=(\tomega_k)_{k} ,Y=(y_k)_{k},\tmu=(\tmu_k)_{k},u) &\longmapsto (\Phi^1,\Phi^2,\Phi^3),
\end{align*}
where
\begin{align*}
 \Phi^1:=(\Phi^1_k)_{1 \leq k \leq K },  \text{ with  } \Phi^1_k(\W,Y,\Modmu,u):&=\re \int_{\Omega} \tilde{R}_k(x) (\bar{u}(x)- \bar{\tilde{R}}(x))dx=    \re \int_{\Omega} \tilde{R}_k(x) \bar{h}(x)     dx        \\ 
  \Phi^2:=(\Phi^2_k)_{1 \leq k \leq K },    \text{ with  } \Phi^2_k(\W,Y,\Modmu,u):&= \re \int_{\Omega} \nabla( \tilde{Q}_{\Modomega_k}(x) )\Psi_k(x) e^{i \tilde{\varphi}_k(x)} (\bar{u}(x)-\bar{\tilde{R}}(x)) dx \\ 
&=  \re \int_{\Omega} \nabla( \tilde{Q}_{\Modomega_k}(x)) \Psi_k(x) e^{i \tilde{\varphi}_k(x)} \bar{h}(x) dx. \\
 \Phi^3:=(\Phi^3_k)_{1 \leq k \leq K },    \text{ with  }\Phi^3_k(\W,Y,\Modmu,u):&= \im \int_{\Omega} \tilde{R}_k(x)( \bar{u}(x)-\bar{\tilde{R}}(x)) dx = \im \int_{\Omega} \tilde{R}_k(x) \bar{h}(x) dx  .
  \end{align*}
Note that, we have $\Phi(W,0,\mu,R(\cdot))=0,$ for $W=(\omega_k)_{1 \leq k \leq K}, \mu:=(\mu_k)_{1 \leq k \leq K}$, and $R(x)=\sum_{k=1}^K R_k(x).$ The Lemma \ref{lem:modu-static} follows from the Implicit Function Theorem if we prove that $d_{\left(\W,Y,\Modmu \right)}\Phi(W,0,\mu,R(\cdot))$ is invertible. For that, we compute the derivatives of $\Phi^1_k, \; \Phi^2_k$ and $\Phi^3_k$ with respect to each $(\omega_k,y_k,\mu_k).$ Note that $h(W,0,\mu,R(\cdot))=0$ and we have
\begin{align*}
  \frac{\partial h}{\partial{\tomega_k}}(W,0,\mu,R(\cdot))&= - \frac{\partial Q_{\omega}}{\partial \omega} \bigg|_{\omega=\omega_k}(x-\alpha_k) \Psi_k(x) e^{ i \varphi_k(x)} , \\
    \frac{\partial h}{\partial{y_k}}(W,0,\mu,R(\cdot)) &=  \nabla Q_{\omega}(x-\alpha_k) \Psi_k(x) e^{ i \varphi_k(x)} , \\
      \frac{\partial h}{\partial{\tmu_k}}(W,0,\mu,R(\cdot)) &= - i  Q_{\omega}(x-\alpha_k) \Psi_k(x) e^{ i \varphi_k(x)} .
\end{align*}

Therefore, we have 
\begin{align*}
\frac{\partial}{\partial \omega_k} \Phi^1_j(W,0,\mu,R)&=- \re \int_{\Omega} Q_{\omega_j}(x-\alpha_j) \Psi_j(x)  e^{i \varphi_j(x)} \frac{\partial Q_{\omega}}{\partial \omega}  \big\vert_{\omega=\omega_k} (x-\alpha_k)  \Psi_k(x) e^{-i \varphi_k(x)} dx. \\
\frac{\partial}{\partial y_k} \Phi^1_j(W,0,\mu,R)&=- \re \int_{\Omega} Q_{\omega_j}(x-\alpha_j) \Psi_j(x)  e^{i \varphi_j(x)}  \nabla Q_{\omega_k}  (x-\alpha_k)  \Psi_k(x) e^{-i \varphi_k(x)} dx. \\
\frac{\partial}{\partial \mu_k} \Phi^1_j(W,0,\mu,R)&=- \im \int_{\Omega} Q_{\omega_j}(x-\alpha_j) \Psi_j(x)  e^{i \varphi_j(x)}   Q_{\omega_k}  (x-\alpha_k)  \Psi_k(x) e^{-i \varphi_k(x)} dx.
\end{align*}
Similar formulas also hold for $\Phi^2_j$ and $\Phi^3_j.$ We first consider the case $k=j$. Using the fact that $ \frac{d }{d\omega} \int Q^2_{\omega}  dx \big\vert_{\omega=\omega_k} >0,$ together with the compact support of $\Psi_k^2-1$ and the exponential decay of $Q_\omega,$ we obtain $\frac{\partial}{\partial \omega_k} \Phi^1_k(W,0,\mu,R) = - a_k + O(e^{-2 \sigma_0 |\alpha_k|}),$ for some $a_k>0.$ Since $Q_{\omega}$ is even function then we have  $\frac{\partial}{\partial y_k} \Phi^1_k(W,0,\mu,R)=O(e^{-2 \sigma_0 |\alpha_k|}).$ Moreover, as $Q_{\omega}$ is real valued function, it follows that  
$\frac{\partial}{\partial \mu_k} \Phi^1_k(W,0,\mu,R)=0.$
Similarly, we obtain for $\Phi^2_k$ and $\Phi^3_k$
\begin{align*}
\frac{\partial}{\partial \omega_k} \Phi^2_k(W,0,\mu,R) &=  O(e^{-2 \sigma_0 |\alpha_k|})    , \quad \quad \quad   \frac{\partial}{\partial \omega_k} \Phi^3_k(W,0,\mu,R) =  O(e^{-2 \sigma_0 |\alpha_k|}), \\
\frac{\partial}{\partial y_k} \Phi^2_k(W,0,\mu,R)&=b_k+ O(e^{-2 \sigma_0 |\alpha_k|}), \quad \frac{\partial}{\partial y_k} \Phi^3_k(W,0,\mu,R)= O(e^{-2 \sigma_0 |\alpha_k|}),  \\
\frac{\partial}{\partial \mu_k} \Phi^2_k(W,0,\mu,R)&=  O(e^{-2 \sigma_0 |\alpha_k|}), \quad  \quad \quad \frac{\partial}{\partial \mu_k} \Phi^3_k(W,0,\mu,R)=c_k+  O(e^{-2 \sigma_0 |\alpha_k|}).
\end{align*}
For the case $k \neq j ,$ we have 
\begin{align*}
 | \frac{\partial}{\partial \omega_k} \Phi^l_j(W,0,\mu,R)  | +  |\frac{\partial}{\partial y_k} \Phi^l_j(W,0,\mu,R)| + |\frac{\partial}{\partial \mu_k} \Phi^l_j(W,0,\mu,R)| \leq C e^{-2\sigma_0 |\alpha_j - \alpha_k|}.
\end{align*}
Thus, the Jacobian of $\Phi=( (\Phi_k^1)_{(1\leq k \leq K)},(\Phi_k^2)_{(1\leq k \leq K)}, (\Phi_k^3)_{(1\leq k \leq K)}) $ at the point $(W,0,\mu,R)$ is not zero and we have $\Phi(W,0,\mu,R(\cdot))=0.$  Then one can apply the implicit function theorem for $u$ in the ball of radius $\varepsilon$  and center $R(x)$ to obtain the existence of the parameters $\widetilde{W}=(\tomega_k)_{1\leq k \leq K} ,Y=(y_k)_{1\leq k \leq K},\tmu=(\tmu_k)_{1\leq k \leq K}$ such that $\Phi(\widetilde{W},Y,\tmu,u)=0$, i.e., there exists $\varepsilon_0 > 0 ,\;  \varepsilon_0 \leq \eta $ and a $C^1$-function 
\begin{align*}
g : B_{H^1_0}(R(\cdot),&\varepsilon ) \longrightarrow B_{R^{K(d+2)}}( (W,0,\mu),\eta) \\
             &u \longmapsto g(u)=\left( \widetilde{W}(u),Y(u),\tmu(u) \right)
\end{align*}
such that $\Phi(\widetilde{W},Y,\tmu)=0$ in $B_{H^1_0}(R(\cdot),\varepsilon ) \times B_{R^{K(d+2)}}( (W,0,\mu),\eta)$ is equivalent to $\left( \widetilde{W}(u),Y(u),\tmu(u) \right)=g(u).$ This concludes the proof of Lemma \ref{lem:modu-static}.
\end{proof}

Note that the previous lemma applies to time independent functions. A consequence of this modulation in the decomposition of fixed $u$ is the following result for the solution $u(t)$ of~$\eqref{eq:defun}$ and satisfying the assumption of Proposition \ref{prop:bt}, in particular \eqref{eq:bt1} for some $\bar{t} \in [T_0,T_n].$

 \begin{corll}
 \label{Cor:Mod-u(t)}
 There exists $C_1 >0$ such that the following holds for all $t \in [\bar{t},T_n]$, for $\bar{t}>T_0$ large, if $u(t,\cdot) \in H^1_0$ satisfies \begin{equation*}
 \label{mod-u(t)-R(t)-L^2 }
 \left\| u(t) - R(t) \right\|_{H^1_0} \leq \alpha_0 e^{-\sigma_0 t } . 
  \end{equation*} 
    Then for $k=1,\ldots, K$ there exits a $C^1$-functions $\tomega_k: [\bar{t},T_n] \longrightarrow (0,\infty) \; ,y_k:[\bar{t},T_n] \longrightarrow \R^d $ and $\tmu_k : [\bar{t},T_n] \longrightarrow \R$ such that if we set  $$h(t,x) = u(t,x)-\widetilde{R}(t,x)  , $$
the following holds for all $t \in [\bar{t},T_n],$
\begin{equation} \label{eq:mod-para}
    \left\|h(t)\right\|_{H^1_0}+ \sum_{k=1}^{K} \left| \Modmu_k (t)- \mu_k \right|+ \sum_{k=1}^{K} |y_k(t)|+\sum_{k=1}^{K} | \Modomega_k(t)-\omega_k| \leq C_1 \left\|u(t)-R(t) \right\|_{L^2},
\end{equation}
Moreover, for all $k=1,\cdots,K$ we have
\begin{equation}
  \label{eq:der-mod-para} 
\left| \frac{d\tomega_k(t) }{dt}  \right|^2 +\left|  \frac{dy_k(t)}{dt}   \right|^2 + \left| \frac{d \tmu_k(t)}{dt} -(\tomega_k(t)-\omega_k) \right|^2 
\leq C_1 \left\| h(t) \right\|_{H^1_0}^2 +  C_1 e^{-2 \sigma_0 t}
\end{equation}       
    
    and
    \begin{align}
\label{eq:orthoC1}
  \re &\int \bar{h}(t,x)  \nabla ( {\widetilde{Q}}_{\tomega_k(t)}(t,x)) \Psi_k(x) e^{i \tvarphi_k(t,x)} dx= 0 ,\; \,     \\ 
  \label{eq:orthoC2} 
  \re &\int \bar{h}(t,x) \widetilde{R}(t,x) dx =0, \\ 
  \label{eq:orthoC3}
  \im& \int \bar{h}(t,x)  \widetilde{R}(t,x) dx  = 0,
    \end{align}

    where, \begin{align*}
&R(t,x)=\sum_{k=1}^K R_k(t,x), \quad R_k(t,x)=Q_{\omega_k}(x-\alpha_k(t))\Psi_k(x) e^{i \varphi_k(t,x)}, \text{with} \quad  \varphi_k(t,x):=\frac 12 (x\cdot v_k)+ \theta_k(t)+\mu_k, \\ 
&\tilde{R}(t,x)=\sum_{k=1}^K \tilde{R}_k(t,x), \quad \tilde{R}_k(t,x)=\tilde{Q}_{\Modomega_k(t)}(t,x) \Psi(x) e^{i \tilde{ \varphi_k  }(t,x) }, \text{with} \quad \tilde{ \varphi_k  }(t,x):=\frac 12 (x\cdot v_k)+ \theta_k(t)+\Modmu_k(t) , \\
&\tilde{Q}_{\Modomega_k}(t,x):=Q_{\Modomega_k}(x-\alpha_k(t)-y_k(t)), \; 
\alpha_k(t)=x_k^0+ tv_k, \; \theta_k(t)= -\frac{1}{4} |v_k|^2 t + \omega_k t.
\end{align*}
 \end{corll}
Note that by $u(T_n)=R(T_n),$ and uniqueness of the decomposition at $t=T_n,$ we have 
\begin{align}
    h(T_n)\equiv 0 \quad \TR(T_n) \equiv R(T_n), \quad \tomega_k(T_n) = \omega_k, \quad y_k(T_n)=0, \quad \tmu_k(T_n)=\mu_k.
\end{align}

\begin{proof}

Recall that we assume \eqref{eq:bt1} holds on $[\bar{t}, T_n]$ for some $\bar{t} \in [T_0,T_n],$ therefore we have $ \left\| u(t) - R(t) \right\|_{H^1_0} \leq \alpha_0 e^{-\sigma_0 t }.$ Applying Lemma \ref{lem:modu-static}, for  $t\in [\bar{t}, T_n]$ and since the map $t \mapsto u(t)$ is continuous in $H^1_0(\Omega),$ we obtain the existence of the continuous function $\tomega_k,\tmu_k$ and $y_k$ on $[\bar{t}, T_n],$ such that \eqref{eq:orthoC1},\eqref{eq:orthoC2} and \eqref{eq:orthoC3} hold. To prove that the functions $\tomega_k,\tmu_k$ and $y_k$ are $C^1$-functions, we use a standard regularization arguments and computation based on the equations of $h$, see \cite{MartelFrank01} for details. We only give the the equation  
of $h(t)$ to justify formally \eqref{eq:mod-para}. It is straightforward to check that $h(t)$ satisfies the equation 
\begin{align} \label{eq:h}
\begin{split}
  i \partial_t h + \mathcal{L} h& =  i \sum_{k=1}^K \bigg( \frac{dy_k(t)}{dt} \cdot \nabla \tQ_{\tomega_k(t)} \Psi_k + ( \frac{d\mu_k(t)}{dt} - ( \tomega_k (t) - \omega_k)  ) \tQ_{\tomega_k(t)} \Psi_k   - i\frac{d \tomega_k(t)}{dt } \frac{d}{d\omega_k} \tQ_{\omega_k(t)} \Psi_k  \bigg) e^{ i \tvarphi_k(t,x)} \\
&- \sum_{k=1}^K \bigg( 2 \nabla \tQ_{\tomega_k(t)}    \nabla \Psi_k  + \tQ_{\tomega_k(t)}  \Delta \Psi_k(x) + i v_k \tQ_{\omega_k(t)} \nabla \Psi_k + \tQ_{\tomega_k(t)}^p \Psi_k(\Psi_k^{p-1} -1)  \bigg)e^{ i \tvarphi_k(t,x)} \\ &+ \alpha(t,x) + \beta(t,x) 
\end{split}
\end{align}
where $\alpha(t,x)$ decay exponentially and $\beta(t,x)$ is a remainder terms on $h$ and 
\begin{align*}
  \mathcal{L}h:=  \Delta h  +  \sum_{k=1}^K   |\TR_k|^{p-1} h + (p-1) | \TR_k|^{p-2} \re( \TR_k \bar{h})
\end{align*}
Using the orthogonality condition \eqref{eq:orthoC1}, \eqref{eq:orthoC2} and \eqref{eq:orthoC3}, we obtain
\begin{align*}
 \im \int \partial_t \overline{h}(t,x) \widetilde{Q}_{\tomega_k}(t,x) \Psi_k(x) dx  & = \displaystyle  \im \int \overline{h}(t,x)  (v_k+\frac{d y_{k}}{dt}(t)) \cdot \nabla \widetilde{Q}_{\tomega_k}(t,x) \Psi_k(x) dx \\
 &- \im \int \overline{h}(t,x)  \frac{d\tomega_k(t)}{dt} \frac{\partial}{\partial \omega_k} \widetilde{Q}_{\omega_k}\big|_{\omega=\tomega_k(t)}(t,x) \Psi_k(x) dx 
 \end{align*}
 
\begin{align*}
\re \int  \partial_t \overline{ h}(t,x)    \widetilde{Q}_{\tomega_k}(t,x)  \, \Psi_k(x) dx &=  \re \int \overline{h}(t,x) (v_{k}+\frac{d y_{k} }{dt}(t))  \cdot \nabla \widetilde{Q}_{\tomega_k} (t,x) \Psi_k(x) dx \\
&- \re \int \overline{h}(t,x)  \frac{d\tomega_k(t)}{dt} \frac{\partial}{\partial \omega_k} \widetilde{Q}_{\omega_k}\big|_{\omega_k=\tomega_k(t)}   \Psi_k(x) dx 
\end{align*}
 
 \begin{align*}
\re \int  \partial_t \overline{ h}(t,x)  \nabla  (\widetilde{Q}_{\tomega_k}(t,x) ) \, \Psi_k(x) dx &=  \re \int \overline{h}(t,x) (v_{k}+\frac{d y_{k} }{dt}(t)) \cdot \nabla(  \nabla (\widetilde{Q}_{\tomega_k} (t,x) )\Psi_k(x)  dx \\
&- \re \int \overline{h}(t,x)  \frac{d\tomega_k(t)}{dt} \frac{\partial}{\partial \omega_k} \nabla (\widetilde{Q}_{\omega_k}\big|_{\omega_k=\tomega_k(t)}(t,x) \Psi_k(x)) dx .
\end{align*}

From the above estimates and the equation \eqref{eq:h} of $h,$ it is straightforward to obtain \eqref{eq:mod-para} by taking the scalar products by $\nabla Q_{\tomega_k(t)} \Psi_k $ and $Q_{\tomega_k(t)} \Psi_k . $ This concludes the proof of Corollary\,\ref{Cor:Mod-u(t)}.
\end{proof}

\subsection{Control of local quantities}
Define a cut-off functions adapted to the solution $u(t).$ Let $\uppsi(x)$ be a smooth function on $\R$ such that 
\begin{align*}
  0 \leq \uppsi \leq 1 , \quad   \uppsi(x)=0, \; \text{for } x \leq -1, \quad \uppsi(x)=1,   \; \text{for } x > 1, \quad \uppsi^{\prime} \geq 0 , \; \text{on } \R 
\end{align*}
and satisfying, for some constants $C>0,$
\begin{align*}
   (\uppsi^{\prime}(x))^2 \leq C \, \uppsi(x), \quad (\uppsi^{\prime \prime}(x))^2 \leq C \,\uppsi^{\prime}(x), \quad \text{for all } x \in \R 
\end{align*}

For $k=2, \cdots, K,$ let $\lambda_k=\frac{1}{2}(v_{k-1,1}+v_{k,1}).$ Let $\Lambda>0$ large enough to be determined later and define 

\begin{align*}
&   \varphi_k(t,x) = \uppsi \bigg( \frac{x_1 - \lambda_k t}{\Lambda}  \bigg)- \uppsi \bigg( \frac{x_1 - \lambda_{k+1} t}{\Lambda}  \bigg), \text{ for } \; k=2, \cdots , K-1,\\
&     \varphi_1(t,x) =1-  \uppsi \bigg( \frac{x_1 - \lambda_2 t}{\Lambda}  \bigg), \quad   \varphi_K(t,x) = \uppsi \bigg( 
     \frac{x_1 - \lambda_K t}{\Lambda}  \bigg)
\end{align*}

Denote by $\mathcal{M}_k(t)= \int |u(t,x)|^2 \varphi_k(t,x) dx.$ and $\mathcal{P}_k(t)= \im \int \nabla u \, \bar{u} \varphi_k(t,x) dx .$

\begin{lem} \label{lem:localMP}
Let $\Lambda>0,$ then there exists $C>0,$ such that if $\Lambda$ and $T_0$ are large enough, then for all $k=2,\cdots,K,$ and for all $t \in [\bar{t},T_n],$
\begin{align} \label{eq:M_k}
  | \mathcal{M}_k(T_n) -\mathcal{M}_k(t ) | + | \mathcal{P}_k(T_n) -\mathcal{P}_k(t ) |  \leq C\frac{\alpha_0^2}{\Lambda} e^{-2 \sigma_0 t}.    
\end{align}
    
\end{lem}
\begin{proof}
We have 
\begin{align*}
\frac{1}{2} \frac{d}{dt} \int |u|^2 \uppsi\left(\frac{x_1-\lambda_k t}{\Lambda} \right)= \frac{1}{\Lambda}
\im \int \partial_{x_1}u \, \bar{u} \uppsi^{\prime}\left(\frac{x_1-\lambda_k t}{\Lambda} \right) - \frac{\lambda_k}{2 \Lambda} \int |u|^2 \uppsi^{\prime}\left(\frac{x_1-\lambda_k t}{\Lambda} \right) .
\end{align*}
Using the properties of the support of $\uppsi_k,$ we obtain 
\begin{align}
\label{eq:dtuloc}
   \left| \frac{d}{dt} \int |u|^2  \uppsi\left(\frac{x_1-\lambda_k t}{\Lambda} \right)  \right| \leq \frac{C}{\Lambda} \int_{\Omega_1} |\partial_{x_1} u |^2 +|u|^2 dx,
\end{align}
where $\Omega_1:=(-L+\lambda_kt, L+\lambda_k t)\times \R^{d-1} \cap \Omega.$ \\

Similarly, we have 
\begin{align*}
 \frac{1}{2} \frac{d}{dt} \im \int \partial_{x_1} u \bar{u}    \uppsi\left(\frac{x_1-\lambda_k t}{\Lambda} \right)= \frac{1}{\Lambda} \int \left( |\partial_{x_1} u |^2 - \frac{p-1}{2(p+1) |u|^{p+1}} \uppsi^{\prime} \left(\frac{x_1-\lambda_k t}{\Lambda} \right)  \right) \\
 - \frac{1}{4 \Lambda^3} \int |u|^2 \uppsi^{'''}\left(\frac{x_1-\lambda_k t}{\Lambda} \right) - \frac{\lambda_k}{2 \Lambda} \im \int \partial_{x_1} u \bar{u} \uppsi^{\prime} \left(\frac{x_1-\lambda_k t}{\Lambda} \right).
\end{align*}
Thus we obtain 
\begin{align*}
   \left| \frac{d}{dt} \im \int \partial_{x_1} u \bar{u} \uppsi\left(\frac{x_1-\lambda_k t}{\Lambda} \right) \right| \leq \frac{C}{\Lambda} \int_{\Omega_1} |\nabla u |^2 + |u|^2 + |u|^{p+1} dx
\end{align*}
Using Sobolev inequality to $u(x) f(x_1-\lambda_k t),$ where $f$ is a $C^1(\R)$ function such that $f(x_1)=1$ for $|x_1|<\Lambda$ and $h(x_1)=0$ for $|x_1|>\Lambda+1,$ we have 

\begin{align*}
   \int_{\Omega_1} |u|^{p+1} \leq C \left( \int_{\Omega_2} |\nabla u|^2 + |u|^2 dx \right)^{\frac{p+1}{2}} ,  
\end{align*}
where $\Omega_2:= (-(L+1)+\lambda_k, (L+1)+\lambda_k) \times \R^{d-1} \cap \Omega.$
Thus, we get 
\begin{align}
\label{eq:dtP1loc}
   \left| \frac{d}{dt} \im \int \partial_{x_1} u \bar{u} \uppsi\left(\frac{x_1-\lambda_k t}{\Lambda} \right) \right| \leq \frac{C}{\Lambda} \int_{\Omega_2} |\nabla u |^2 + |u|^2  dx+ \frac{C}{\Lambda} \left( \int_{\Omega_2} |\nabla u|^2 + |u|^2 dx \right)^{\frac{p+1}{2}}.
\end{align}

Now, for $j=2,\cdots,d,$ we have 
\begin{align*}
    \frac{1}{2} \frac{d}{dt} \im \int \partial_{x_j} u \, \bar{u} \uppsi \left(\frac{x_1-\lambda_k t}{\Lambda} \right) dx = \frac{1}{\Lambda} \re \int \partial_{x_j} u \partial_{x_1} \bar{u} \uppsi^{\prime} \left(\frac{x_1-\lambda_k t}{\Lambda} \right) + \frac{\lambda_k}{2 \Lambda} \im \int \partial_{x_j} u \bar{u} \uppsi^{\prime} \left(\frac{x_1-\lambda_k t}{\Lambda} \right),
\end{align*}
which yields, 
\begin{align}\label{eq:dtPjloc}
\left| \frac{d}{dt} \im \int \partial_{x_j} u \, \bar{u}  \uppsi \left(\frac{x_1-\lambda_k t}{\Lambda} \right) \right|    \leq \frac{C}{\Lambda} \int_{\Omega_1} |\nabla u|^2 + |u|^2 .
\end{align}
Next, by writing $u(t)=R(t)+u(t)-R(t),$ and using \eqref{eq:bt1} we have 
\begin{align*}
   \int_{\Omega_2} |\nabla u |^2 + |u|^2 dx & \leq  2 \int_{\Omega_2} | \nabla R(t)|^2 +|R(t)|^2 dx+ 2 \left\| u(t) - R(t) \right\|_{H^1_0} \\
   &\leq C \int_{\Omega_2} | \nabla R(t)|^2 +|R(t)|^2   dx  + \alpha_0^2 e^{-2 \sigma_0 t}.  
\end{align*}
Using the definitions of $\Psi_k$, $\lambda_k$ and $ \sigma_0,$ we have 
\begin{align*}
  \int_{\Omega_2} | \nabla R(t)|^2 +|R(t)|^2   dx \leq C e^{-8\sqrt{\sigma_0} (\sqrt{\sigma_0} t - \Lambda)} \leq   C e^{-4 \sigma_0 t }   
\end{align*}
Taking $T_0$ and $\Lambda$ so that $\sqrt{\sigma_0} T_0 \geq 2 \Lambda.$ \\ 

Using \eqref{eq:dtuloc}, \eqref{eq:dtP1loc} and \eqref{eq:dtPjloc}, with definition of $\mathcal{M}_k(t)$ and $\mathcal{P}_k(t)$ we obtain 

\begin{align}
\left|  \frac{d}{dt} \mathcal{M}_k(t) \right|  + \left|  \frac{d}{dt} \mathcal{P}_k(t) \right| \leq C \frac{\alpha_0^2}{\Lambda} e^{-2 \sigma_0 t }.
\end{align}
Integrating between $t$ and $T_n$ we obtain the desired results.

\end{proof}

\begin{lem}
\label{lem:omega_k}
    For any $t \in [\bar{t},T_n],$ we have
    \begin{align} \label{eq:omega_k}
       \left| \tomega_k(t) - \omega_k \right| \leq C \left\| h \right\|_{H^1_0}^2 + C (\frac{\alpha_0^2}{\Lambda} + 1) e^{-2 \sigma_0 t }  
    \end{align}
\end{lem}
\begin{proof}
Expanding $u(t)=\TR(t) + h(t),$ using the properties of the support of $\varphi_k$ and $\Psi_k$
with the orthogonality conditions \eqref{eq:orthoC2}, \eqref{eq:orthoC3}, we obtain 
\begin{align*}
 \mathcal{M}_k(t)= \int |u(t)|^2 \varphi_k(t) dx &= \int \tQ^2_{\tomega_k(t)}  + \int | h(t)|^2 \varphi_k(t) dx+\int \tQ^2_{\tomega_k(t)} ( \Psi_k^2-1) + O(e^{-2 \sigma_0 t} ) \\
=& \int \tQ^2_{\tomega_k(t)}  + \int | h(t)|^2 \varphi_k(t) dx+ O(e^{-2 \sigma_0 t} ) 
\end{align*}
By Lemma \ref{lem:localMP}, we have 
\begin{align*}
     | \mathcal{M}_k(T_n) -\mathcal{M}_k(t ) |   \leq C\frac{\alpha_0^2}{\Lambda} e^{-2 \sigma_0 t}.
\end{align*}
Using the fact that $\tomega_k(T_n)=\omega_k$ and $h(T_n)=0,$ we have
\begin{align*}
  \left| \int Q^2_{\tomega_k(t)} - Q^2_{\omega_k}  \right| \leq C \left\| h(t)\right\|_{H^1_0} + C \left( \frac{\alpha_0^2}{\Lambda} +1 \right) e^{-2 \sigma_0 t } 
\end{align*}
Recall that we assume that $\int Q^2_{\omega} = \omega^{\frac{2}{p-1}-\frac{d}{2}} \int Q^2.$ Thus, for $\tomega_k(t)-\omega_k$ small enough we obtain 

\begin{align*}
 \int Q^2_{\tomega_k(t)} - Q^2_{\omega_k}  & = ( \tomega_k^{\frac{2}{p-1}-\frac{d}{2}}(t)   - \omega_k^{\frac{2}{p-1}-\frac{d}{2}} ) \int Q^2 \\
 &= \left( \frac{2}{p-1}-\frac{d}{2} \right) \omega_k^{\frac{2}{p-1}-\frac{d}{2}} ( \tomega_k(t) - \omega_k) \int Q^2 + O(\tomega_k(t)-\omega_k)^{1+\varepsilon},
\end{align*}
where $\varepsilon>0$ and $ \frac{2}{p-1}-\frac{d}{2} >0 .$ Using \eqref{eq:mod-para}, we obtain \eqref{eq:omega_k}. Note that for general nonlinearity $f(|u|^2 u)$, we assume that $\frac{d}{d\omega} \int Q^2_{\omega} dx \bigg|_{\omega=\omega_k}>0,$ using similar argument, we obtain the desired result. 
\end{proof}
Next, we control $\| h(t)\|_{H^1_0}$ and the modulation parameters. The estimate for $h(t)$ relies on the previous bounds, together with the almost conserved quantities $\mathcal{M}_k(t)$ and $\mathcal{P}_k(t)$, as well as the energy functional, which links the localized mass, momentum, and energy. We denote by \begin{align*}
   \mathfrak{J}(t):= \sum_{k=1}^{K} \bigg\{  (\omega_k+ \frac{|v_k|^2}{4}) \int |u(t)|^2 \varphi_k(t) - v_k \im \int \nabla u(t) \bar{u}(t) \varphi_k(t) dx     \bigg\}
\end{align*}
and we set $$ \mathcal{G}(t):=E(u(t))+ \mathfrak{J}(t).$$ 
Note that since the coefficients in $\mathfrak{J}(t)$ depend on the parameters of each solitary wave, it is necessary to localize around each of them. Let us mention that this functional is used to investigate the stability of one solitary waves with nonzero velocity and to address the multi-soliton case, where the nonzero velocity condition plays a crucial role; see \cite{MR2228459}.

\begin{lem}
For all $t \in [\bar{t},T_n],$ we have
\begin{align} \label{eq:G(t)}
  \mathcal{G}(t)&=   \sum_{k=1}^{K} \bigg( E(\tQ_{\omega_k(t)}) + \omega_k \int |\tQ_{\omega_k(t)}|^2 dx \bigg)+  H_K(h(t),h(t))  +O(e^{-2\sigma_0 t }) + \left\| h \right\|_{H^1_0 }^2 \beta\left( \left\| h \right\|_{H^1_0 } \right),
\end{align}
where $\beta(h) \to 0 ,$ as $h \to 0 ,$ and
\begin{align*}
 H_K(h(t),h(t)):&= \int | \nabla h (t) |^2 dx  - \sum_{k=1}^{K} \int  f(| \TR_k(t)|^{2}) |h(t)|^2 + 2  f^{\prime}(|\TR_k(t)|^2) (\re(\TR_k(t) \bar{h}(t)))^2 dx  \\
 &+ \sum_{k=1}^K 
    \bigg\{(\omega_k+ \frac{|v_k|^2}{4})    \int |h(t)|^2 \varphi_k(t) dx - v_k \cdot \left( \im \int \nabla h(t) \bar{h}(t) \varphi_k(t) dx \right) \bigg\}
\end{align*}
Moreover, with the orthogonality conditions for $h(t),$ that is, \eqref{eq:orthoC1}, \eqref{eq:orthoC2} and \eqref{eq:orthoC3}, we have for $\lambda>0,$ 

\begin{equation} \label{eq:coef:prop}
    H_K(h(t),h(t)) \geq \lambda \left\| h(t) \right\|_{H^1_0} . 
\end{equation}

\end{lem}

\begin{proof}

Expanding $u(t) = \TR(t) + h(t),$ we have
\begin{align*}
    E(u(t))&=E(\TR(t))- 2  \re \int ( \Delta \TR   + f(|\TR|^{2})  \TR ) \bar{h} dx + \int | \nabla h |^2 dx - \int  f(| \TR|^{2}) |h|^2 \\
    &+ 2  f^{\prime}(|\TR|^2) (\re(\TR \bar{h}))^2 dx + \left\| h \right\|_{H^1_0 }^2 \beta\left( \left\| h \right\|_{H^1_0 } \right)   
\end{align*}

Using the fact that $\TR_k$ are exponentially decaying, we have for $k\neq j,$
\begin{align}
\label{eq:R_jR_k}
  \int | \TR_j \TR_k | dx +    \int | \TR_j  \nabla \TR_k | dx +   \int | \nabla \TR_j \nabla \TR_k | dx < C e^{-2\theta_0 t }
\end{align}
Since $f(0)=0,$ and $|F(s)| < C s^2$ and $|f(s)|< C s$ in a neighborhood of zero. Therefore, we have
\begin{align*}
 E(u(t))&=\sum_{k=1}^{K} E(\TR_k(t))- 2 \sum_{k=1}^{K} \re \int ( \Delta \TR_k   + f(|\TR_k|^{2})  \TR_k ) \bar{h} dx+ O(e^{-2 \sigma_0 t }) + \int | \nabla h |^2 dx  \\ 
 &- \sum_{k=1}^{K} \int  f(| \TR_k|^{2}) |h|^2 
    + 2  f^{\prime}(|\TR_k|^2) (\re(\TR_k \bar{h}))^2 dx + \left\| h \right\|_{H^1_0 }^2 \beta\left( \left\| h \right\|_{H^1_0 } \right)       
\end{align*}

Notice that, form definition of $\TR$ we have
\begin{align*}
  \Delta \TR_k &= \Delta \tQ_{\tomega_k(t)} \Psi_k  e^{i \tvarphi_k(t,x) }+ i v_k \nabla  \tQ_{\tomega_k(t)} \Psi_ke^{i \tvarphi_k(t,x) }  - \frac{|v_k|^2}{4}  \tQ_{\tomega_k(t)} \Psi_k  e^{i \tvarphi_k(t,x) } \\
  &+  2 \nabla \tQ_{\tomega_k(t)} \nabla \Psi_k  e^{i \tvarphi_k(t,x) } +  2  \tQ_{\tomega_k(t)} \Delta \Psi_k  e^{i \tvarphi_k(t,x) } +  \tQ_{\tomega_k(t)} v_k \nabla \Psi_k  e^{i \tvarphi_k(t,x) } 
\end{align*}
and \begin{align*}
   f(|\TR_k|^{2})  \TR_k =    f(|\tQ_{\tomega_k(t)}|^{2})  \tQ_{\tomega_k(t)} \Psi_k e^{i \tvarphi_k(t,x) } +    f(|\tQ_{\tomega_k(t)}|^{2}) \tQ_{\tomega_k(t)} \Psi_k(  f(|\Psi_k^{p-1}|^{2})  -1) 
\end{align*}

Therefore, we have 
\begin{align*}
  \Delta \TR_k   +  f(|\TR_k|^{2})  \TR_k &= \tomega_k(t) \TR_k(t)  + i v_k \nabla ( \tQ_{\tomega_k(t)} \Psi_k ) e^{i \tvarphi_k(t,x) }  - \frac{|v_k|^2}{4}  \tQ_{\tomega_k(t)} \Psi_k  e^{i \tvarphi_k(t,x) } + \gamma(t,x), \\
\end{align*}
where 
\begin{align*}
    \gamma(t,x)=   &  2 \nabla \tQ_{\tomega_k(t)} \nabla \Psi_k  e^{i \tvarphi_k(t,x) } +  2  \tQ_{\tomega_k(t)} \Delta \Psi_k  e^{i \tvarphi_k(t,x) } + f(|\tQ_{\tomega_k(t)}|^{2}) \tQ_{\tomega_k(t)} \Psi_k 
  ( f(|\Psi_k|^{2})  -1) 
\end{align*}
Notice that $\gamma $ has a compact support, by definition of $\Psi.$ \\
Using the fact that $\nabla \Psi_k$ and $\Psi_k^2-1$ have compact support, we obtain
\begin{align}
\label{eq:E1}
\nonumber E(u(t))&=\sum_{k=1}^{K} \bigg( E(\tQ_{\omega_k(t)}) + \frac{|v_k|^2}{4} \int |\tQ_{\omega_k(t)}|^2 dx \bigg)- 2 \sum_{k=1}^{K} \tomega_k(t) \re \int  \TR_k(t)  \bar{h}(t) dx \\ 
 &-2  \sum_{k=1}^{K} \re \int  (i v_k \cdot \nabla ( \tQ_{\tomega_k(t)} \Psi_k )- \frac{|v_k|^2}{4}  \tQ_{\tomega_k(t)} \Psi_k   ) e^{i \tvarphi_k(t,x) }  \bar{h}(t) dx + O(e^{-2 \sigma_0 t }) \\ \nonumber
& + \int | \nabla h |^2 dx  - \sum_{k=1}^{K} \int  f(| \TR_k|^{2}) |h|^2 
    + 2  f^{\prime}(|\TR_k|^2) (\re(\TR_k \bar{h}))^2 dx + \left\| h \right\|_{H^1_0 }^2 \beta\left( \left\| h \right\|_{H^1_0 } \right)     .     
\end{align}
Next, we turn to estimate $\mathfrak{J}(t),$
\begin{align*}
   \mathfrak{J}(t)= \sum_{k=1}^{K} \bigg\{  (\omega_k+ \frac{|v_k|^2}{4}) \int |u(t)|^2 \varphi_k(t) - v_k \im \int \nabla u(t) \bar{u}(t) \varphi_k(t) dx     \bigg\}
\end{align*}
Expanding the mass term, and using the properties of $\varphi_k, $ we have
\begin{align*}
 \int |u(t)|^2 \varphi_k(t) dx &= \int |\TR(t) |^2 \varphi_k(t) dx +  2 \re \int \TR(t) \bar{h} \varphi_k(t)  dx +  \int |h(t)|^2  \varphi_k(t) dx\\
 &=  \int | \TR_k(t) |^2 dx +  2 \re \int \TR_k(t) \bar{h} \varphi_k(t)  dx +  \int |h(t)|^2 \varphi_k(t) dx   + O(e^{-2 \sigma_0 t }) \\
 & = \int | \tQ_{\tomega_k(t)} |^2 dx + \int  \tQ_{\tomega_k(t)}^2 (\Psi_k^2-1)  dx +  2 \re \int \TR_k(t) \bar{h}(t)  dx \\
 & +  \int |h(t)|^2 \varphi_k(t) dx   + O(e^{-2 \sigma_0 t }). 
\end{align*}
Since $\Psi_k^2-1$ has compact has a compact support, then we have 
\begin{align*}
 \int |u(t)|^2 \varphi_k(t) dx & = \int | \tQ_{\tomega_k(t)} |^2 dx +  2 \re \int \TR_k(t) \bar{h}(t)  dx \\
 & +  \int |h(t)|^2 \varphi_k(t) dx   + O(e^{-2 \sigma_0 t }). 
\end{align*}
Using \eqref{eq:R_jR_k}, the momentum contribution satisfies, 
\begin{align*}
  \im \int \nabla u \bar{u} \varphi_k(t) dx&= \im \int \nabla \TR(t) \overline{\TR}(t) \varphi_k(t)- \im \int \TR(t) h(t) \varphi_k^{\prime}(t) dx  -2 \im \int \nabla \TR(t) \bar{h}(t) \varphi_k(t) dx \\
  &+ \im \int \nabla h(t) \bar{h}(t) \varphi_k(t) dx \\
  &=  \im \int \nabla \TR_k(t) \overline{\TR_k} (t) - \im \int \TR(t) h(t) \varphi_k(t)dx   -2  \im \int \nabla \TR_k(t) \bar{h}(t) dx
   \\
   &+ \im \int \nabla h(t) \bar{h}(t) \varphi_k(t) dx\\
   &= \frac{v_k}{2} \int |\tQ_{\omega_k(t)}|^2 dx  +O(e^{-2\sigma_0 t }) -2  \im \int \nabla \TR_k(t) \bar{h}(t) dx+ \im \int \nabla h(t) \bar{h}(t) \varphi_k(t) dx.
\end{align*}
Therefore, we obtain
\begin{align}
\label{eq:J2} 
\mathfrak{J}(t)&= \sum_{k=1}^K \bigg\{ (\omega_k+ \frac{|v_k|^2}{4}) \left(  \int | \tQ_{\tomega_k(t)} |^2 dx  + 2   \re \int \TR_k(t) \bar{h}(t)  dx +  \int |h(t)|^2 \varphi_k(t) dx     \right) \\ \nonumber
&- \frac{|v_k|^2}{2}   \int |\tQ_{\omega_k(t)}|^2 dx  + 2  v_k  \im \int \nabla \TR_k(t) \bar{h}(t) dx
   - v_k \im \int \nabla h(t) \bar{h}(t) \varphi_k(t) dx +O(e^{-2\sigma_0 t })
  \bigg\}   .
\end{align}

Using \eqref{eq:E1} and \eqref{eq:J2}, we obtain 
\begin{align*}
  \mathcal{G}(t)&=   \sum_{k=1}^{K} \bigg( E(\tQ_{\omega_k(t)}) + \omega_k \int |\tQ_{\omega_k(t)}|^2 dx \bigg)+  H_K(h(t),h(t))- 2 \sum_{k=1}^{K} \tomega_k(t) \re \int  \TR_k(t)  \bar{h}(t) dx  \\
  &-2  \sum_{k=1}^{K} \re \int  (i v_k \cdot \nabla (\tQ_{\tomega_k(t)} \Psi_k )- \frac{|v_k|^2}{4}  \tQ_{\tomega_k(t)} \Psi_k    e^{i \tvarphi_k(t,x) }  \bar{h}(t) dx\\
&+ \sum_{k=1}^K \bigg\{ (\omega_k+ \frac{|v_k|^2}{4})   2   \re \int \TR_k(t) \bar{h}(t)  dx      + 2  v_k  \im \int \nabla \TR_k(t) \bar{h}(t) dx   \bigg\}  +O(e^{-2\sigma_0 t }) + \left\| h \right\|_{H^1_0 }^2 \beta\left( \left\| h \right\|_{H^1_0 } \right)
\end{align*}
Using the orthogonality conditions \eqref{eq:orthoC2}, we have 
\begin{align*}
    \sum_{k=1}^{K}  \re \int  \TR_k(t)  \bar{h}(t) dx=  \re \int  \TR(t)  \bar{h}(t) dx=0
\end{align*}
and
\begin{align*}
&-2  \sum_{k=1}^{K} \re \int  (i v_k \cdot \nabla (\tQ_{\tomega_k(t)} \Psi_k )- \frac{|v_k|^2}{4}  \tQ_{\tomega_k(t)} \Psi_k   ) e^{i \tvarphi_k(t,x) }  \bar{h}(t) dx   +  \frac{|v_k|^2}{2}      \re \int \TR_k(t) \bar{h}(t)  dx  \\
& + 2  v_k  \im \int \nabla \TR_k(t) \bar{h}(t) dx =  -2  \sum_{k=1}^{K} \im \int   v_k \cdot \nabla (\tQ_{\tomega_k(t)} \Psi_k ) e^{i \tvarphi_k(t,x) } \bar{h}(t) dx + \sum_{k=1}^{K} \re \int  |v_k|^2 \TR_k(t)  \bar{h}(t) dx \\
  &+ 2 \sum_{k=1}^{K}   v_k \im \int \nabla (\tQ_{\tomega_k(t)} \Psi_k )  e^{i \tvarphi_k(t,x) } \bar{h}(t) dx - \sum_{k=1}^{K} \re \int |v_k|^2 \TR_k(t) \bar{h}(t) dx =0
\end{align*}

Therefore, we have 
\begin{align*}
  \mathcal{G}(t)&=   \sum_{k=1}^{K} \bigg( E(\tQ_{\omega_k(t)}) + \omega_k \int |\tQ_{\omega_k(t)}|^2 dx \bigg)+  H_K(h(t),h(t))  +O(e^{-2\sigma_0 t }) + \left\| h \right\|_{H^1_0 }^2 \beta\left( \left\| h \right\|_{H^1_0 } \right)
\end{align*}
Next, recall that from \cite[Section 2]{MR820338} we have, for $\omega >0$ close to $\omega_0>0,$ 
 \begin{align*}
  \left| E(Q_{\omega_0}) + \omega_0 \int Q_{\omega_0}^2dx - ( E(Q_{\omega}) + \omega_0 \int Q_{\omega}^2dx  \right|   \leq C \left| \omega_0 -\omega \right|^2 .
 \end{align*}
Using the above property, we obtain the desired result \eqref{eq:G(t)}. Note that the proof of \eqref{eq:coef:prop} is standard and based on Lemma \ref{lem:modu-static}, see \cite[Lemma 2.4]{Ou19}. It also requires localization arguments similar to those in \cite{MartelMerleTai02} and \cite[Appendix B]{MR2228459}, which remain valid for $h \in H^1_0(\Omega)$. Indeed, $h$ can be extended to a function in $H^1(\mathbb{R}^3)$ by setting $h(x) = 0$ for $x \in \Theta$.
\end{proof}

\begin{lem}
\label{lem:control-mod-para}
    For any $t \in [\bar{t},T_n],$ we have 
\begin{align*}
\left\| h(t)\right\|_{H^1_0}^2 + | \tomega_k(t) - \omega_k| + |y_k(t)|^2 + | \tmu_k(t) - \mu_k|^2 \leq C \left( \frac{\alpha_0^2}{\Lambda}  + 1\right) e^{-2 \sigma_0 t }.        
\end{align*}
\end{lem}

\begin{proof}
In the case $k=1,$  one can use similar arguments as for the construction of a solitary wave solution for the \NNls equation, see \cite{Ou19}. Notice that by Lemma \ref{lem:localMP}, we have for $t \in [\bar{t},T_n],$ 
\begin{align*}
  \left| \mathfrak{J}(t) - \mathfrak{J}(T_n) \right| \leq C \frac{\alpha_0^2}{\Lambda} e^{-2 \sigma_0 t}.
\end{align*}
From conservation of the energy $E(u(t))$ and the above estimate, we have
\begin{align*}
   \mathcal{G}(t) \leq \mathcal{G}(T_n) + C \frac{\alpha_0^2}{\Lambda} e^{-2 \sigma_0 t }
\end{align*}
Using the fact that $\tomega_k(T_n)=\omega_k$ and $h(T_n)=0,$ we obtain 
\begin{align*}
  H_K(h(t),h(t)) \leq C   \left\| h \right\|_{H^1_0 }^2 \beta\left( \left\| h \right\|_{H^1_0 } \right) + C \frac{\alpha_0^2}{\Lambda} e^{-2 \sigma_0 t}.
\end{align*}
By the coercivity property \eqref{eq:coef:prop}, we obtain 
\begin{align*}
   \lambda \left\| h(t) \right\|_{H^1_0}^2 \leq  C   \left\| h \right\|_{H^1_0 }^2 \beta\left( \left\| h \right\|_{H^1_0 } \right) + C \frac{\alpha_0^2}{\Lambda} e^{-2 \sigma_0 t},
\end{align*}
Note that by \eqref{eq:mod-para}, we have  $\left\| h \right\|_{H^1_0}$ is small enough provided $\alpha_0 e^{-\sigma_0 T_0}$ small enough. 
Then, it follows that 
\begin{align*}
   \lambda \left\| h(t) \right\|_{H^1_0}^2 \leq   C \frac{\alpha_0^2}{\Lambda} e^{-2 \sigma_0 t},  
\end{align*}
where $C$ is independent of $\alpha_0.$ \\

Next, by Lemma \ref{lem:omega_k} and the above estimate we have
\begin{align*}
         \left| \tomega_k(t) - \omega_k \right| \leq  C (\frac{\alpha_0^2}{\Lambda} + 1) e^{-2 \sigma_0 t }  
\end{align*}
Finally, we use \eqref{eq:der-mod-para} to control the rest of the modulation parameters 
\begin{align*}
\left|  \frac{dy_k(t)}{dt}   \right|+ \left| \frac{d \tmu_k(t)}{dt} \right| 
\leq C_1 \left\| h(t) \right\|_{H^1_0} +  C e^{- \sigma_0 t}+ |\tomega_k(t)-\omega_k| \leq C \sqrt{ (\frac{\alpha_0^2}{\Lambda} + 1) } e^{- \sigma_0 t} 
\end{align*}
Integrating between $T_n$ and $t \in [\bar{t},T_n)$, and using the fact that $y_k(T_n)=0$ and $\tmu_k(T_n)=\mu_k,$ we obtain 
\begin{align*}
   \left|  y_k(t)  \right|^2 + \left| \tmu_k(t)\right|^2
\leq  C  (\frac{\alpha_0^2}{\Lambda} + 1) e^{-2 \sigma_0 t}  
\end{align*}
This concludes the proof of Lemma \ref{lem:control-mod-para}.

\end{proof}

\subsection{Proof of Proposition \ref{prop:bt}} Finally, we conclude the proof of Proposition \ref{prop:bt} using the estimates the from Lemma \ref{lem:control-mod-para}. For $ t \in [\bar{t},T_n]$, we have 
\begin{align*}
  \left\| R(t) - \TR(t) \right\|_{H^1_0}^2 \leq C \sum_{k=1}^K \left( |y_k(t)|^2 + |\tomega_k(t) - \omega_k|^2 + |\tmu_k - \mu_k|^2 \right)  \leq C ( \frac{\alpha_0^2}{\Lambda}+ 1 ) e^{-2 \alpha_0 t},
\end{align*}
and thus 
\begin{align*}
  \left\| u(t)- R(t) \right\|_{H^1_0}^2 \leq 2 \left\| h(t) \right\|_{H^1_0}^2 + 2 \left\| \TR(t) - R(t) \right\|_{H^1_0}^2  \leq C ( \frac{\alpha_0^2}{\Lambda}+ 1 ) e^{-2 \alpha_0 t},
\end{align*}
where $C$ is a constant that does not depend on $\alpha_0.$ Choosing $\alpha_0^2> 32C$ and $\Lambda=\alpha_0^2,$ we obtain for $T_0 $ large enough 
\begin{align*}
  \left\| u(t)- R(t) \right\|_{H^1_0}^2 \leq \frac{\alpha_0^2}{16}  e^{-2 \sigma_0 t}.
\end{align*}
Then, for all $t \in [\bar{t},T_n], $ we have $\left\| u(t)- R(t) \right\|_{H^1_0} \leq \frac{\alpha_0}{4}  e^{- \sigma_0 t}.$ This completes the proof of Proposition \ref{prop:bt}.

\bibliographystyle{acm}
\bibliography{references.bib}

\begin{thebibliography}{10}

\bibitem{AnRa08}
{\sc Anton, R.}
\newblock Global existence for defocusing cubic {NLS} and {G}ross-{P}itaevskii
  equations in three dimensional exterior domains.
\newblock {\em J. Math. Pures Appl. (9) 89}, 4 (2008), 335--354.

\bibitem{BeLi83a}
{\sc Berestycki, H., and Lions, P.-L.}
\newblock Nonlinear scalar field equations. {I}. {E}xistence of a ground state.
\newblock {\em Arch. Rational Mech. Anal. 82}, 4 (1983), 313--345.

\bibitem{BlairSmithSogge2012}
{\sc Blair, M.~D., Smith, H.~F., and Sogge, C.~D.}
\newblock Strichartz estimates and the nonlinear {S}chr\"{o}dinger equation on
  manifolds with boundary.
\newblock {\em Math. Ann. 354}, 4 (2012), 1397--1430.

\bibitem{BuGeTz04a}
{\sc Burq, N., G{\'e}rard, P., and Tzvetkov, N.}
\newblock On nonlinear {S}chr\"odinger equations in exterior domains.
\newblock {\em Ann. Inst. H. Poincar\'e Anal. Non Lin\'eaire 21}, 3 (2004),
  295--318.

\bibitem{CaLi82}
{\sc Cazenave, T., and Lions, P.-L.}
\newblock Orbital stability of standing waves for some nonlinear
  {S}chr\"odinger equations.
\newblock {\em Comm. Math. Phys. 85}, 4 (1982), 549--561.

\bibitem{MR3124722}
{\sc Combet, V.}
\newblock Multi-existence of multi-solitons for the supercritical nonlinear
  {S}chr\"{o}dinger equation in one dimension.
\newblock {\em Discrete Contin. Dyn. Syst. 34}, 5 (2014), 1961--1993.

\bibitem{MR2815738}
{\sc C\^ote, R., Martel, Y., and Merle, F.}
\newblock Construction of multi-soliton solutions for the {$L^2$}-supercritical
  g{K}d{V} and {NLS} equations.
\newblock {\em Rev. Mat. Iberoam. 27}, 1 (2011), 273--302.

\bibitem{ThomasOSvetlana22}
{\sc Duyckaerts, T., Landoulsi, O., and Roudenko, S.}
\newblock Threshold solutions in the focusing 3{D} cubic {NLS} equation outside
  a strictly convex obstacle.
\newblock {\em J. Funct. Anal. 282}, 5 (2022), Paper No. 109326, 55.

\bibitem{ThomasYang24}
{\sc Duyckaerts, T., and Yang, J.~U.}
\newblock Dispersive estimates for wave and {S}chr\" odinger equations with a
  potential in non-trapping exterior domains.
\newblock {\em Preprint arXiv:2401.12608\/} (2024).

\bibitem{GidasNirenberg79}
{\sc Gidas, B., Ni, W.~M., and Nirenberg, L.}
\newblock Symmetry and related properties via the maximum principle.
\newblock {\em Comm. Math. Phys. 68}, 3 (1979), 209--243.

\bibitem{Ivanovici07}
{\sc Ivanovici, O.}
\newblock Precised smoothing effect in the exterior of balls.
\newblock {\em Asymptot. Anal. 53}, 4 (2007), 189--208.

\bibitem{Ivanovici10}
{\sc Ivanovici, O.}
\newblock On the {S}chr{\"o}dinger equation outside strictly convex obstacles.
\newblock {\em Analysis \& PDE 3}, 3 (2010), 261--293.

\bibitem{MR2683754}
{\sc Ivanovici, O., and Planchon, F.}
\newblock On the energy critical {S}chr\"odinger equation in {$3D$}
  non-trapping domains.
\newblock {\em Ann. Inst. H. Poincar\'e Anal. Non Lin\'eaire 27}, 5 (2010),
  1153--1177.

\bibitem{killip2015riesz}
{\sc Killip, R., Visan, M., and Zhang, X.}
\newblock Riesz transforms outside a convex obstacle.
\newblock {\em International Mathematics Research Notices 2016}, 19 (2015),
  5875--5921.

\bibitem{Kwong89}
{\sc Kwong, M.~K.}
\newblock Uniqueness of positive solutions of {$\Delta u-u+u^p=0$} in {${\bf
  R}^n$}.
\newblock {\em Arch. Rational Mech. Anal. 105}, 3 (1989), 243--266.

\bibitem{Ou19}
{\sc Landoulsi, O.}
\newblock Construction of a solitary wave solution of the nonlinear focusing
  {S}chr\"{o}dinger equation outside a strictly convex obstacle in the
  {$L^2$}-supercritical case.
\newblock {\em Discrete Contin. Dyn. Syst. A 41}, 2 (2021), 701--746.

\bibitem{OL22-blow-up}
{\sc Landoulsi, O.}
\newblock On blow-up solutions to the nonlinear {S}chr\"{o}dinger equation in
  the exterior of a convex obstacle.
\newblock {\em Dyn. Partial Differ. Equ. 19}, 1 (2022), 1--22.

\bibitem{OsvetlanaKai23}
{\sc Landoulsi, O., Roudenko, S., and Yang, K.}
\newblock Interaction with an obstacle in the 2{D} focusing nonlinear
  {S}chr\"{o}dinger equation.
\newblock {\em Adv. Comput. Math. 49}, 5 (2023), Paper No. 71, 44.

\bibitem{DongSmithZhang2012}
{\sc {Li}, D., {Smith}, H., and {Zhang}, X.}
\newblock {Global well-posedness and scattering for defocusing energy-critical
  NLS in the exterior of balls with radial data.}
\newblock {\em {Math. Res. Lett.} 19}, 1 (2012), 213--232.

\bibitem{Maris02}
{\sc Mari\c~s, M.}
\newblock Existence of nonstationary bubbles in higher dimensions.
\newblock {\em J. Math. Pures Appl. (9) 81}, 12 (2002), 1207--1239.

\bibitem{MartelFrank01}
{\sc Martel, Y., and Merle, F.}
\newblock Instability of solitons for the critical generalized {K}orteweg-de
  {V}ries equation.
\newblock {\em Geom. Funct. Anal. 11}, 1 (2001), 74--123.

\bibitem{MartelMerle06}
{\sc Martel, Y., and Merle, F.}
\newblock Multi solitary waves for nonlinear {S}chr\"odinger equations.
\newblock {\em Ann. Inst. H. Poincar\'e{} C Anal. Non Lin\'eaire 23}, 6 (2006),
  849--864.

\bibitem{MR2271697}
{\sc Martel, Y., and Merle, F.}
\newblock Multi solitary waves for nonlinear {S}chr\"odinger equations.
\newblock {\em Ann. Inst. H. Poincar\'e Anal. Non Lin\'eaire 23}, 6 (2006),
  849--864.

\bibitem{MartelMerleTai02}
{\sc Martel, Y., Merle, F., and Tsai, T.-P.}
\newblock Stability and asymptotic stability in the energy space of the sum of
  {$N$} solitons for subcritical g{K}d{V} equations.
\newblock {\em Comm. Math. Phys. 231}, 2 (2002), 347--373.

\bibitem{MR2228459}
{\sc Martel, Y., Merle, F., and Tsai, T.-P.}
\newblock Stability in {$H^1$} of the sum of {$K$} solitary waves for some
  nonlinear {S}chr\"odinger equations.
\newblock {\em Duke Math. J. 133}, 3 (2006), 405--466.

\bibitem{McLeod93}
{\sc McLeod, K.}
\newblock Uniqueness of positive radial solutions of {$\Delta u+f(u)=0$} in
  {${\bf R}^n$}. {II}.
\newblock {\em Trans. Amer. Math. Soc. 339}, 2 (1993), 495--505.

\bibitem{Morawetz61}
{\sc Morawetz, C.~S.}
\newblock The decay of solutions of the exterior initial-boundary value problem
  for the wave equation.
\newblock {\em Comm. Pure Appl. Math. 14\/} (1961), 561--568.

\bibitem{MoraRalstonStraussCorec78}
{\sc Morawetz, C.~S., Ralston, J.~V., and Strauss, W.~A.}
\newblock Correction to: ``{D}ecay of solutions of the wave equation outside
  nontrapping obstacles'' ({C}omm. {P}ure {A}ppl. {M}ath. {\bf 30} (1977), no.
  4, 447--508).
\newblock {\em Comm. Pure Appl. Math. 31}, 6 (1978), 795.

\bibitem{PlVe09}
{\sc Planchon, F., and Vega, L.}
\newblock Bilinear virial identities and applications.
\newblock {\em Ann. Sci. \'Ec. Norm. Sup\'er. (4) 42}, 2 (2009), 261--290.

\bibitem{MR820338}
{\sc Weinstein, M.~I.}
\newblock Lyapunov stability of ground states of nonlinear dispersive evolution
  equations.
\newblock {\em Comm. Pure Appl. Math. 39}, 1 (1986), 51--67.

\bibitem{Wilcox59}
{\sc Wilcox, C.~H.}
\newblock Spherical means and radiation conditions.
\newblock {\em Arch. Rational Mech. Anal. 3\/} (1959), 133--148.

\end{thebibliography}

\end{document}